\documentclass[english]{article}
\usepackage[T1]{fontenc}
\usepackage[latin9]{inputenc}
\usepackage{color}
\usepackage{amsthm}
\usepackage{amsmath}
\usepackage{amssymb}
\usepackage{esint}
\usepackage{latexsym}
\usepackage{graphicx}

\usepackage{csquotes}
\MakeOuterQuote{"}
\makeatletter
\theoremstyle{plain}
	\newtheorem{thm}{\protect\theoremname}[section]
  \theoremstyle{plain}
  \newtheorem{cor}[thm]{\protect\corollaryname}
  \newtheorem*{cor*}{\protect\corollaryname}
  \newtheorem{lem}[thm]{\protect\lemmaname}
  \newtheorem*{lem*}{\protect\lemmaname}
	\newtheorem{prop}[thm]{\protect\propositionname}
	\newtheorem*{prop*}{\protect\propositionname}
  \theoremstyle{remark}
  \newtheorem{rem}{\protect\remarkname}
\usepackage{hyperref}
\usepackage{nameref}
\hypersetup{pdfpagemode=UseNone}

\DeclareMathOperator{\tr}{Tr}
\DeclareMathOperator{\iid}{Id}
\DeclareMathOperator{\MD}{MD}
\DeclareMathOperator{\Proj}{Proj}

\makeatletter
\let\orgdescriptionlabel\descriptionlabel
\renewcommand*{\descriptionlabel}[1]{%
  \let\orglabel\label
  \let\label\@gobble
  \phantomsection
  \edef\@currentlabel{#1}%
  \let\label\orglabel
  \orgdescriptionlabel{#1}%
}
\makeatother

\makeatother

\usepackage{babel}
	\providecommand{\corollaryname}{Corollary}
  \providecommand{\lemmaname}{Lemma}
  \providecommand{\remarkname}{Remark}
	\providecommand{\theoremname}{Theorem}
\providecommand{\propositionname}{Proposition}
\begin{document}

\title{Regularity of solutions to fully nonlinear elliptic and parabolic free boundary problems}
\date{}

\author{E. Indrei and A. Minne}

\def\signei{\bigskip\begin{center} {\sc Emanuel Indrei\par\vspace{3mm}Center for Nonlinear Analysis\\  
Carnegie Mellon University\\
Pittsburgh, PA 15213, USA\\
email:} {\tt eindrei@msri.org }
\end{center}}

\def\signam{\bigskip\begin{center} {\sc Andreas Minne \par\vspace{3mm}
Department of Mathematics\\
KTH, Royal Institute of Technology\\
100 44 Stockholm, Sweden\\
email:} {\tt minne@kth.se}
\end{center}}

\maketitle

\begin{abstract}
We consider fully nonlinear obstacle-type problems of the form 
\begin{equation*}
\begin{cases}
F(D^{2}u,x)=f(x) & \text{a.e. in }B_{1}\cap\Omega,\\
|D^{2}u|\le K & \text{a.e. in }B_{1}\backslash\Omega,
\end{cases}
\end{equation*}
where $\Omega$ is an unknown open set and $K>0$. In particular, structural conditions on $F$ are presented which ensure that $W^{2,n}(B_1)$ solutions achieve the optimal $C^{1,1}(B_{1/2})$ regularity when $f$ is H\"older continuous. Moreover, if $f$ is positive on $\overline B_1$, Lipschitz continuous, and $\{u\neq 0\} \subset \Omega$, then we obtain local $C^1$ regularity of the free boundary under a uniform thickness assumption on $\{u=0\}$. Lastly, we extend these results to the parabolic setting.    
\end{abstract}

\makeatletter
\def\blfootnote{\gdef\@thefnmark{}\@footnotetext}
\makeatother

\blfootnote{E. Indrei acknowledges support from the Australian Research Council, US NSF Grant DMS-0932078 administered by the Mathematical Sciences Research Institute in Berkeley, CA, and US NSF PIRE Grant OISE-0967140 administered by the Center for Nonlinear Analysis at Carnegie Mellon University.}

\section{Introduction}

Obstacle-type problems appear in several mathematical disciplines such as minimal surface theory, potential theory, mean field theory of superconducting vortices, optimal control, fluid filtration in porous media, elasto-plasticity, and financial mathematics \cite{Crev, C1, CS, Rod, MR0440187}. The classical obstacle problem involves minimizing the Dirichlet energy on a given domain in the space of square integrable functions with square integrable gradient constrained to remain above a fixed obstacle function and with prescribed boundary data. Due to the structure of the Dirichlet integral, this minimization process leads to the free boundary problem 

\begin{equation*} \label{lap0}
\Delta u = f \chi_{\{u > 0\}} \hskip .1in \text{in} \hskip .1in B_1,
\end{equation*}
where $B_1 \subset \mathbb{R}^n$ is the unit ball centered at the origin. A simple one-dimensional example shows that even if $f \in C^\infty$, $u$ is not more regular than $C^{1,1}$. If the right-hand side is Lipschitz continuous, then the Harnack inequality may be used to show that $u$ achieves this optimal regularity. 

An obstacle-type problem is a free boundary problem of the form 
\begin{equation} \label{lap}
\Delta u = f \chi_\Omega \hskip .1in \text{in} \hskip .1in B_1,
\end{equation}
where $\Omega$ is an (unknown) open set. If $\Omega = \{u \neq 0\}$ and $f$ is Lipschitz continuous, monotonicity formulas may be used to prove $C^{1,1}$ regularity of $u$. Nevertheless, this method fails when $f$ is H\"older continuous. Recently, a harmonic analysis technique was developed in \cite{ALS} to prove optimal regularity under the weakest possible assumption on $f$: if $f$ is Dini-continuous, then $u$ is uniformly $C^{1,1}$ in $B_{1/2}$, where the bound on the Hessian depends on $\|u\|_{L^\infty(B_{1})}$.

Fully nonlinear analogs of \eqref{lap} have been considered by several researchers. The case $$F(D^2 u) = f \chi_\Omega \hskip .1in \text{in} \hskip .1in B_1$$  has been studied in \cite{L} for $\Omega = \{u > 0\}$ and in \cite{LS} when $\Omega = \{u \neq 0\}$. Moreover, a fully nonlinear version of the method in \cite{ALS} was developed in \cite{2012arXiv1212.5809F} and applied to                   
\begin{equation*}
\begin{cases}
F(D^{2}u)=1 & \text{a.e. in }B_{1}\cap\Omega,\\
|D^{2}u|\le K & \text{a.e. in }B_{1}\backslash\Omega,
\end{cases}
\end{equation*}
where $\Omega$ is an open set, $K>0$, and $u \in W^{2,n}(B_1)$. The  idea is to replace the projection on second-order harmonic polynomials carried out in \cite{ALS} with a projection involving the BMO estimates in \cite{MR1978880}. Subject to certain structural conditions on $F$, this tool is employed to prove that $u \in C^{1,1}$ in $B_{1/2}$ and, under a standard thickness assumption, that the free boundary is locally $C^1$. Moreover, the general structure of the equation enables the authors to recover previous regularity results (e.g. when $\Omega=\{u\neq 0\})$ and address nonlinear parabolic free boundary problems in the case when the elliptic operator does not depend explicitly on the spatial variable \cite{2013arXiv1309.0782F}.  

Our main result is Theorem \ref{c11} and establishes optimal regularity for the free boundary problem 

\begin{equation}
\begin{cases}
F(D^{2}u,x)=f(x) & \text{a.e. in }B_{1}\cap\Omega,\\
|D^{2}u|\le K & \text{a.e. in }B_{1}\backslash\Omega,
\end{cases}\label{eq:main}
\end{equation}
where $\Omega$ is an open set, $K>0$, $f$ is H\"older continuous, and under certain structural conditions on $F$ (see \S \ref{setup}). As a direct consequence, we obtain optimal regularity for general operators $F(D^2u, Du, u, x)$ and thereby address a problem discussed by Figalli and Shahgholian \cite[Remark 1.1]{2012arXiv1212.5809F}, see Corollary \ref{genops}. Free boundary problems of this type appear in the mean field theory of superconducting vortices \cite[Introduction]{CS}. 

The underlying principle in the proof is to locally apply Caffarelli's elliptic regularity theory \cite{MR1005611} to rescaled variants of \eqref{eq:main} in order to obtain a bound on $D^2u$. The main difficulty lies in verifying an average $L^n$ decay of the right-hand side in question. However, one may exploit that $u \in C^{1,\alpha}(B_1)$, $D^2u$ is bounded in $B_{1}\backslash\Omega$, and the BMO estimates in \cite{MR1978880} to prove that locally around a free boundary point, the coincidence set $B_{1}\backslash\Omega$ decays fast enough to ensure the $L^n$ decay. Our assumptions on the structure of $F$ involve conditions which enable us to utilize standard tools such as the maximum principle and Evans-Krylov theorem.

Moreover, once we establish that $u \in C^{1,1}$ in $B_{1/2}$, the corresponding regularity theory for the free boundary follows in a standard way through the classification of blow-up solutions and is carried out in \S \ref{frebd}. Indeed, non-degeneracy holds if $f$ is positive on $\overline B_1$ and $\{|\nabla u| \neq 0\} \subset \Omega$. Moreover, blow-up solutions around thick free boundary points  are half-space solutions, and this fact combines with a directional monotonicity result to yield $C^{1}$ regularity of the free boundary, see Theorem \ref{fbreg} for a precise statement.

Finally, we generalize the above-mentioned results to the parabolic setting in \S \ref{para} by considering the free boundary problem 
\begin{equation*}
\begin{cases}
\mathcal{H}(u(X),X)=f(X) & \text{a.e. in }Q_{1}\cap\Omega,\\
|D^{2}u|\le K & \text{a.e. in }Q_{1}\backslash\Omega,
\end{cases}
\end{equation*}
where $X=(x,t) \in \mathbb{R}^n \times \mathbb{R}$, $\mathcal{H}(u(X),X):=F(D^{2}u(X),X)-\partial_{t}u(X)$, $Q_1$ is the parabolic cylinder $B_{1}(0)\times(-1,0)$, $\Omega\subset Q_{1}$
is some unknown set, and $K>0$. 

\subsubsection*{Acknowledgements} 

\noindent We thank Alessio Figalli, Henrik Shahgholian, and John Andersson for enlightening discussions on obstacle-type free boundary problems. Moreover, we wish to thank Henrik Shahgholian and John Andersson for their valuable remarks on a preliminary version of this paper.   

\subsection{Setup} \label{setup}

In what follows, we record the structural conditions on the operator $F$ that will be employed throughout this paper. The first three conditions are well known in the study of free boundary problems and provide tools such as the maximum principle and Evans-Krylov theorem. The last condition, which we denote by \ref{eq:H4}, is the new ingredient which controls the oscillation of the operator in the spatial variable and enables the application of Caffarelli's regularity theory in our general framework, see Remarks \ref{r1} \& \ref{r2}. Moreover, we note that throughout the paper the constants of proportionality in our estimates may change from line to line while still being denoted by the same symbol $C$.  
    
\begin{description}
\item [{(H1)\label{hyp:0@0}}] $F(0,x)=0$ for all $x\in\Omega$.
\item [{(H2)\label{hyp:uniformellipticity}}] The operator $F$ is uniformly
elliptic with ellipticity constants $\lambda_{0}$, $\lambda_{1}>0$
such that
\[
\mathcal{P}^{-}(M-N)\le F(M,x)-F(N,x)\le\mathcal{P}^{+}(M-N)\qquad\forall x\in\Omega,
\]
where $M$ and $N$ are symmetric matrices and $\mathcal{P}^{\pm}$
are the Pucci operators
\[
\mathcal{P}^{-}(M):=\inf_{\lambda_{0}\iid\le N\le\lambda_{1}\iid}\tr NM,\qquad\mathcal{P}^{+}(M):=\sup_{\lambda_{0}\iid\le N\le\lambda_{1}\iid}\tr NM.
\]

\item [{(H3)\label{hyp:concavity}}] $F(M,x)$ will be assumed to be concave
or convex in $M$ for all $x$ in $\Omega$.
\end{description}

\begin{description}
\item [{(H4)\label{eq:H4}}] 
\begin{equation*} 
|F(M,x)-F(M,y)|\le C|M||x-y|^\alpha,
\end{equation*}
for some $\alpha \in (0,1]$. 
\end{description}

\begin{rem} Note that \ref{hyp:0@0} is not restrictive since we can work with $G(M,x):=F(M,x)-F(0,x)$ which fulfills \ref{hyp:uniformellipticity} with the same ellipticity constants as well as \ref{hyp:concavity} and \ref{eq:H4}. The uniform ellipticity also implies Lipschitz regularity,
\begin{equation}
|F(M,x)-F(N,x)|\le\max\{|\mathcal{P}^{-}(M-N)|,|\mathcal{P}^{+}(M-N)|\}\le n\lambda_{1}|M-N|.\label{eq:lipschitz}
\end{equation}
In particular, 
\begin{equation}
|F(M,x)-F(M,y)|\le|F(M,x)-F(0,x)|+|F(M,y)-F(0,y)|\le2n\lambda_{1}|M|.\label{eq:boundeddifffromlip}
\end{equation}
\end{rem}

\begin{rem} \label{r1} 
Let
\begin{equation*}
\beta(x) = \sup_{M \in \mathcal{S}} \frac{|F(M,x)-F(M,0)|}{|M|} \quad\text{and}\quad \tilde \beta(x) = \sup_{M \in \mathcal{S}} \frac{|F(M,x)-F(M,0)|}{|M|+1},
\end{equation*}
 where $\mathcal{S}$ is the space of symmetric matrices. Note that \ref{eq:H4} implies the H\"older continuity of both $\beta$ and $\tilde \beta$. 
\end{rem}

\begin{rem} \label{r2}
\ref{eq:H4} is equivalent to saying that if $$\bar \beta(x,y) = \sup_{M \in \mathcal{S}} \frac{|F(M,x)-F(M,y)|}{|M|},$$ then $\bar \beta$ is dominated (up to a constant) by $|x-y|^\alpha$ for some $\alpha \in (0,1]$. When $y=0$, this is equivalent to asking that $\beta$ is H\"older continuous at the origin which comes up in \cite{MR1005611} and \cite{MR1351007}. In fact, one may weaken this to a suitable integrability condition.     
\end{rem}

\section{\texorpdfstring{$C^{1,1}$}{C1,1} regularity}

In this section, we prove optimal regularity for $W^{2,n}(B_1)$ solutions of the free boundary problem \eqref{eq:main}:  
\begin{thm} \label{c11}
Let $f \in C^{\alpha}(B_1)$ be a given function and $\Omega$ a domain such that $u:B_{1}\to\mathbb{R}$ is a $W^{2,n}(B_1)$ solution of 
\begin{equation*}
\begin{cases}
F(D^{2}u,x)=f(x) & \text{a.e. in }B_{1}\cap\Omega,\\
|D^{2}u|\le K & \text{a.e. in }B_{1}\backslash\Omega.
\end{cases}
\end{equation*}
Assume $F$ satisfies \ref{hyp:0@0}-\ref{eq:H4}. Then there exists a constant $\overline{C}>0$, depending on $\|u\|_{W^{2,n}(B_{1})}$, $\|f\|_{L^\infty(B_1)}$, the dimension, and the ellipticity constants such that
\[
|D^{2}u|\le\overline{C},\qquad\text{a.e. in }B_{1/2}.
\]
\end{thm}

Since $W^{2,n}(B_1)$ solutions of \eqref{eq:main} are $C^{1,\alpha}(B_1)$, one may utilize the above theorem to deduce an optimal regularity result for more general operators and thereby address a problem discussed by Figalli and Shahgholian \cite[Remark 1.1]{2012arXiv1212.5809F}: 

\begin{cor} \label{genops}
Let $f \in C^{\alpha}(B_1)$ be a given function and $\Omega$ a domain such that  $u:B_{1}\to\mathbb{R}$ is a $W^{2,n}(B_1)$ solution of 
\begin{equation*}
\begin{cases}
F(D^{2}u,Du, u, x)=f(x) & \text{a.e. in }B_{1}\cap\Omega,\\
|D^{2}u(x)|\le K & \text{a.e. in }B_{1}\backslash\Omega,
\end{cases}
\end{equation*}
and assume that: $F(0,v,t,x)=0$ for all $v \in \mathbb{R}^n$, $t \in \mathbb{R}$, and $x \in \Omega$; $F$ satisfies \ref{hyp:0@0}-\ref{hyp:concavity}  in the matrix variable (keeping all other variables fixed); and,  
\begin{equation*} 
|F(M,w_1, s_1, x_1)-F(M, w_2, s_2, x_2)|\le C|M|(|w_1-w_2|^{\alpha_1}+|s_1-s_2|^{\alpha_2}+|x_1-x_2|^{\alpha_3}),
\end{equation*}
for some $\alpha_i \in (0,1]$. Then there exists a 
constant $\overline{C}>0$, depending on $\|u\|_{W^{2,n}(B_{1})}$, $\|f\|_{L^\infty(B_1)}$, the dimension, and the ellipticity constants such that
\[
|D^{2}u|\le\overline{C},\qquad\text{a.e. in }B_{1/2}.
\]
\end{cor}

\begin{proof}
Define $$\tilde F(M,x):= F(M, Du(x), u(x), x),$$ and simply note that the assumptions on $F$ together with the fact that $u \in C^{1,\alpha}(B_1)$ imply that $\tilde F$ satisfies the assumptions of Theorem \ref{c11}.   
\end{proof}

\noindent {\bf Standing assumptions:}
Unless otherwise stated, we let $x_{0}\in B_{1/2} \cap \overline{\Omega}$ and assume without loss of generality that $u(x_{0})=|\nabla u(x_{0})|=0$ (otherwise we can replace $u(x)$ with $\tilde u(x):=u(x)-u(x_0)-\nabla u(x_0)\cdot (x-x_0)$). 

\noindent Moreover, set
\[
	A_{r}(x_{0}):=\frac{(B_{r}(x_{0})\backslash\Omega) -x_0}{r}=B_{1}\backslash((\Omega-x_{0})/r).
\]
Whenever we refer to a solution $u$ of \eqref{eq:main}, it is implicit that $u \in W^{2,n}$ and $F$ satisfies \ref{hyp:0@0}-\ref{eq:H4}.

The theorem will be established through several key lemmas. The first step consists of finding a suitable approximation for the Hessian of $u$ at $x_0$ through the following projection lemma.      
    
\begin{lem} 
\label{lem:boundedBMOminimizer} Let $f\in L^\infty(B_1)$ and $u$ be a solution to \eqref{eq:main}.
Then there exists a constant $C=C(\|u\|_{W^{2,n}(B_{1})},\|f\|_{L^\infty(B_1)},n,\lambda_{0})>0$
such that 
\[
\min_{F(P, x_0)=f(x_{0})}\fint_{B_{r}(x_{0})}|D^{2}u(y)-P|^{2}\, dy\le C,\qquad\forall r\in(0,1/4).
\]
\end{lem}

\begin{proof}
Let $Q_{r}(x_{0}):=(D^{2}u)_{r,x_{0}}=\fint_{B_{r}(x_{0})}D^{2}u(y)\, dy$ and note that for $t \in \mathbb{R}$, the ellipticity and boundedness of $F$ implies  
\begin{align*}
 & \mathcal{P}^{-}(t\iid)\le F(Q_{r}(x_{0})+t\iid,x_{0})-F(Q_{r}(x_{0}),x_{0})\le\mathcal{P}^{+}(t \iid)\\
\Rightarrow & \lambda_{0}tn-C\le F(Q_{r}(x_{0})+t\iid,x_{0})\le\lambda_{1}tn+C.
\end{align*}
Thus, there exists $\xi_{r}(x_{0}) \in \mathbb{R}$ such that $F(Q_{r}(x_{0})+\xi_{r}(x_{0})\iid,x_{0})=f(x_{0})$ (by continuity). With this in mind,  
\begin{align*}  
&\phantom{\le}\min_{F(P, x_0)=f(x_{0})}\fint_{B_{r}(x_{0})}|D^{2}u(y)-P|^{2}\, dy\\
& \le\fint_{B_{r}(x_{0})}|D^{2}u(y)-Q_{r}(x_{0})-\xi_{r}(x_{0})\iid|^{2}\, dy\\
 & \le2\fint_{B_{r}(x_{0})}|D^{2}u(y)-Q_{r}(x_{0})|^{2}\, dy+2\xi_{r}(x_{0})^{2}\\
 & \le 2 C_{BMO} + 2\xi_{r}(x_{0})^{2},
\end{align*}
where we have used the BMO estimate in \cite{2012arXiv1212.5809F}. It remains to find a uniform bound on $\xi_{r}(x_{0})$: applying \eqref{eq:lipschitz}, \eqref{eq:boundeddifffromlip}, Hölder's inequality, and the BMO estimate again, we obtain
\begin{align*}
|F(Q_{r}(x_{0}),x_{0})| & = \bigg|\fint_{B_{r}(x_{0})}F(Q_{r}(x_0)-D^{2}u(y)+D^{2}u(y),x_{0})\, dy\bigg|\\
 & \le\fint_{B_{r}(x_{0})}|F(D^{2}u(y),x_{0})|+n\lambda_{1}|D^{2}u(y)-Q_{r}(x_{0})|\, dy\\
 & \le\fint_{B_{r}(x_{0})}\big(|F(D^{2}u(y),x_{0})-F(D^{2}u(y),y)|+|F(D^{2}u(y),y)|\\
 & \qquad+n\lambda_{1}|D^{2}u(y),x_{0})-Q_{r}(x_{0})|\,\big) dy\\
 & \le \fint_{B_{r}(x_{0})}|F(D^{2}u(y),x_{0})-F(D^{2}u(y),y)|\, dy\\
 & \qquad+\max\{||f||_\infty, n\lambda_1 K\}\\
 & \qquad  +n\lambda_{1}\sqrt{\fint_{B_{r}(x_{0})}|D^{2}u(y)-Q_{r}(x_{0})|^2\, dy}\\
 & \le 2n\lambda_1\|D^{2}u\|_{W^{2,n}(B_{1})}+\max\{||f||_\infty, n\lambda_1 K\}+C_{BMO} =:C.
\end{align*}
Thus, 
\begin{align*}
 & \mathcal{P}^{-}(\xi_{r}(x_{0})\iid)\le F(Q_{r}(x_{0})+\xi_{r}(x_{0})\iid,x_{0})-F(Q_{r}(x_{0}),x_{0})\le\mathcal{P}^{+}(\xi_{r}(x_{0})\iid)\\
\Rightarrow & \lambda_{0}\xi_{r}(x_{0})n-C\le F(Q_{r}(x_{0})+\xi_{r}(x_{0})\iid,x_{0})\le\lambda_{1}\xi_{r}(x_{0})n+C\\
\Rightarrow & \lambda_{0}\xi_{r}(x_{0})n-C\le f(x_{0})\le\lambda_{1}\xi_{r}(x_{0})n+C.
\end{align*}
In particular, $|\xi_{r}(x_{0})| \le \frac{\Vert f\Vert_{\infty}+C}{\lambda_{0}n}$ and this concludes the proof.
\end{proof}

\noindent In what follows, let $P_{r}(x_{0})$ denote any minimizer of $$\min_{F(P,x_0)=f(x_{0})}\fint_{B_{r}(x_{0})}|D^{2}u(y)-P|^{2}\, dy,$$
for $r\in(0,1/4)$. Lemma \ref{lem:boundedBMOminimizer} and the triangle inequality readily imply that the growth of $P_r(x_0)$ is controlled in $r$:

\begin{cor}\label{cor:2rTOr}
\label{lem:BMO-consequence}Let $f\in L^\infty(B_1)$ and $u$ be a solution to \eqref{eq:main}.
Then there exists a constant $C_{0}=C_{0}(\|D^{2}u\|_{W^{2,n}(B_{1})},\|f\|_{L^\infty(B_1)},n,\lambda_{0})$
such that 
\begin{equation*}
|P_{2r}(x_{0})-P_{r}(x_{0})|\le C_{0} \qquad\forall r\in(0,1/8).
\end{equation*}
\end{cor}

\begin{rem}
We note that \ref{eq:H4} is not needed in the proofs of Lemma \ref{lem:boundedBMOminimizer} and Corollary \ref{cor:2rTOr}. 
\end{rem}

\noindent Next we verify that $P_r(x_0)$ is a suitable approximation to $D^2u(x_0)$.

\begin{lem} \label{lemo}
\label{lem:r2bound}Let $f\in L^\infty(B_1)$ and $u$ be a solution to \eqref{eq:main}. Then
there exists a constant $C_{1}=C_1(K,\|f\|_{L^\infty(B_1)},n,\lambda_0, \lambda_{1}, ||u||_{W^{2,n}(B_1)})$
such that 
\begin{equation*}
\sup_{x\in B_{r}(x_{0})}\bigg|u(x)-\frac{1}{2}\left\langle P_{r}(x_{0})(x-x_{0}),(x-x_{0})\right\rangle \bigg|\le C_{1}r^{2} \label{eq:r2bound} \qquad\forall r\in(0,1/8).
\end{equation*}
\end{lem}

\begin{proof}
Assume without loss of generality that $F$ is concave (otherwise, consider $\tilde F(M,x):=-F(-M,x)$ and $v=-u$) and define $$u_{r,x_{0}}(y):=\frac{u(ry+x_{0})}{r^{2}}-\frac{1}{2}\langle P_{r}(x_{0})y,y\rangle,$$
$$G(Q):=G(Q,x_{0}):=F(P_{r}(x_{0})+Q,x_{0})-f(x_{0}).$$ Then, $G(0)=0$ and 
\begin{align*}
G(D^2 u_{r,x_{0}}(y))&=F(D^2u(ry+x_0),x_0)-f(x_0)\\
&=F(D^2u(ry+x_0),ry+x_0)-f(x_0)+h(y),
\end{align*}
where $h(y):=F(D^2u(ry+x_0),x_0)-F(D^2u(ry+x_0),ry+x_0)$. Thus, $u_{r,x_{0}}$ solves
\begin{equation*}
\begin{cases}
G(D^{2}u_{r,x_{0}}(y))=f(ry+x_{0})-f(x_{0})+h(y), & \text{in } B_{1} \setminus A_{r}(x_{0}),\\
G(D^{2}u_{r,x_{0}}(y))= F(D^2u(ry+x_0),ry+x_0)-f(x_0)+h(y), & \text{in } A_{r}(x_{0}).
\end{cases}
\end{equation*}
Next note that if $ry+x_0 \notin \Omega$, then $F(D^2u(ry+x_0),ry+x_0)$ is bounded, so by letting 
\begin{equation*} 
\begin{cases}
\phi(y):=f(ry+x_{0})-f(x_{0})& \text{in } B_{1} \setminus A_{r}(x_{0}),\\
\phi(y):= F(D^2u(ry+x_0),ry+x_0)-f(x_0)& \text{in } A_{r}(x_{0}),
\end{cases}
\end{equation*}
it follows that $\phi$ has an $L^\infty$ bound depending only on the given data and  

\begin{equation} \label{G}
G(D^{2}u_{r,x_{0}}(y))=\phi(y)+h(y) \hskip .2in  \text{a.e. in }B_{1}.
\end{equation}   
Moreover, $\overline{u}_{r,x_{0}}(y):=u_{r,x_{0}}(y)-(u_{r,x_{0}})_{1,0}-y\cdot(\nabla u_{r,x_{0}})_{1,0}$
solves the same equation. Since $$u_{r,x_{0}}(0)=|\nabla u_{r,x_{0}}(0)|=0,$$  (recall $u(x_0)=|\nabla u(x_0)|=0$ by assumption) it follows that 
\begin{align*}
(u_{r,x_{0}})_{1,0}&=-\overline{u}_{r,x_{0}}(0),\\
(\nabla u_{r,x_{0}})_{1,0}&=-\nabla\overline{u}_{r,x_{0}}(0),
\end{align*}
and we may write $u_{r,x_{0}}(y)=\overline{u}_{r,x_{0}}(y)-\overline{u}_{r,x_{0}}(0)-y\cdot\nabla\overline{u}_{r,x_{0}}(0)$.
Next we wish to apply Theorem 2 in \cite{MR1005611}. First note that our assumptions on $F$ imply the required interior a priori estimates for $G$; moreover, $G$ has no spatial dependence so it remains to verify the $L^n$ condition of $\phi+h$. Since $\phi$ has an $L^\infty$ bound depending only on the given data, we need to verify it solely for $h$.  Indeed, let $s \leq 1$ and note that thanks to \ref{eq:H4},  
\begin{equation} \label{cond}
\int_{B_s} |h(y)|^n dy \leq (rs)^{\alpha n} \int_{B_s} |D^2u(ry+x_0)|^n dy \leq C||u||_{W^{2,n}(B_1)} s^{\alpha n}.
\end{equation}
Therefore, applying the theorem yields
\begin{align}
\Vert u_{r,x_{0}}\Vert_{L^{\infty}(\overline{B}_{1/2})} & =\Vert\overline{u}_{r,x_{0}}-\overline{u}_{r,x_{0}}(0)-y\cdot\nabla\overline{u}_{r,x_{0}}(0)\Vert_{L^{\infty}(\overline{B}_{1/2})} \nonumber \\
 & \le C(\Vert\overline{u}_{r,x_{0}}\Vert_{L^{\infty}(B_{1})}+1) \label{sdu},
\end{align}
where $C$ does not depend on $r$. Moreover, due to the concavity of $G$ (which is inherited from $F$),
there is a linear functional $L$ so that $L(Q)\ge G(Q,x_0)$ and $L(0)=0$ (this linear functional depends on $x_0$). In particular, $$L(D^{2}\overline{u}_{r,x_{0}}(y))\ge G(D^{2}\overline{u}_{r,x_{0}}(y),x_0)=\phi(y)+h(y),$$ a.e.
in $B_{1}$ (recall \eqref{G}); this fact together with \eqref{cond} and Corollary 9.20 in \cite{gilbarg:01} applied to the subsolutions  $\overline{u}_{r,x_{0}}^+$ and $\overline{u}_{r,x_{0}}^-$ implies 

\begin{align*}
\Vert\overline{u}_{r,x_{0}}\Vert_{L^{\infty}(B_{1})} &\le C\Vert\overline{u}_{r,x_{0}}\Vert_{L^{2}(B_{1})}+\Vert \phi+h\Vert_{L^{n}(B_{1})}\\
&\le C\Vert\overline{u}_{r,x_{0}}\Vert_{L^{2}(B_{1})}+C(K,\|f\|_{L^\infty(B_1)},n,\lambda_0, \lambda_{1}, \|u\|_{W^{2,n}(B_1)}),
\end{align*}
and applying the Poincaré inequality twice yields 
\begin{align*}
\Vert\overline{u}_{r,x_{0}}\Vert_{L^{2}(B_{1})} & \le C\Vert D^{2}u_{r,x_{0}}\Vert_{L^{2}(B_{1})}=C\fint_{B_{r}(x_{0})}|D^{2}u(y)-P_r(x_0)|^{2}\, dy\le C,
\end{align*}
where Lemma \ref{lem:boundedBMOminimizer} is used in the last inequality. This combined with \eqref{sdu} implies $$\Vert u_{r,x_{0}}\Vert_{L^{\infty}(B_{1/2})}\le C;$$
thus,
\[
\sup_{B_{r/2}(x_{0})}\bigg|\frac{u(x)-\frac{1}{2}\left\langle P_{r}(x_{0})(x-x_{0}),(x-x_{0})\right\rangle }{r^{2}}\bigg|\le C.
\]
The result now follows by replacing $r/2$ with $r$ and utilizing Corollary \ref{lem:BMO-consequence}.

\end{proof}
\begin{lem}
\label{lem:fastdecay}Let $f\in C^0(B_1)$ and $u$ be a solution to \eqref{eq:main}. Then there exists a constant $M=M(K,\|f\|_{L^\infty(B_1)},n,\lambda_{0})$ such
that, for any $r\in(0,1/8)$,
\[
|A_{r/2}(x_{0})|\le\frac{|A_{r}(x_{0})|}{2^{n}}
\]
if $|P_{r}(x_{0})|>M$.\end{lem}

\begin{proof}
Let $u_{r,x_{0}}(y):=\frac{u(ry+x_{0})}{r^{2}}-\frac{1}{2}\langle P_{r}(x_{0})y,y\rangle$
and $$\tilde G(Q,y):=F(P_{r}(x_{0})+Q,ry+x_0)-f(x_{0}).$$ Remark 1 below Theorem 8.1 in \cite{MR1351007} implies the existence of a solution $v_{r,x_{0}}$ to the equation 
\begin{equation} 
\begin{cases}
\tilde G(D^{2}v_{r,x_{0}}(y),y)=f(ry+x_{0})-f(x_{0}) & \text{in }B_{1},\\
v_{r,x_{0}}=u_{r,x_{0}} & \text{on }\partial B_{1};
\end{cases}\label{eq:r-nofreebdry}
\end{equation}
set $$w_{r,x_{0}}:=u_{r,x_{0}}-v_{r,x_{0}},$$ and note that by definition $$\tilde G(D^{2}u_{r,x_{0}}(y),y)=F(D^2 u(ry+x_0),ry+x_0)-f(x_0).$$ Therefore, 
\begin{align*}
\tilde G(&D^{2}u_{r,x_{0}}(y),y)-\tilde G(D^{2}v_{r,x_{0}}(y),y)\\
&=\big(F(D^2u(ry+x_0),ry+x_0)-f(x_0)\big)-(f(ry+x_{0})-f(x_{0}))\\
&=\big(F(D^2u(ry+x_0), ry+x_0)-f(ry+x_{0})\big)\chi_{A_r(x_0)}\\
&=:\tilde \phi(y)\chi_{A_r(x_0)},
\end{align*} 
where $\tilde \phi \in L^\infty(B_1)$. Combining this information with \eqref{hyp:uniformellipticity} and the definition of $\tilde G$ yields 
\begin{align*}
\mathcal{P}^{-}(D^{2}w_{r,x_{0}}(y)) & \le \tilde G(D^{2}u_{r,x_{0}}(y),y)- \tilde G(D^{2}v_{r,x_{0}}(y),y)\\
 & =\tilde \phi(y) \chi_{A_{r}(x_{0})}\le\mathcal{P}^{+}(D^{2}w_{r,x_{0}}).
\end{align*}
Since $\tilde \phi \in L^\infty(B_1)$ with bounds depending only on the given data and $A_r(x_0)$ is relatively closed in $B_1$ (recall that $\Omega$ is open), we may apply the ABP estimate to obtain 
\begin{equation}
\Vert w_{r,x_{0}}\Vert_{L^{\infty}(B_{1})}\le C(K,f,n,\lambda_{0},\lambda_{1})|A_{r}(x_{0})|^{1/n}.\label{eq:ABPconsequence}
\end{equation}
Since \ref{eq:H4} holds, we may combine Remark 3 following Theorem 8.1 in \cite{MR1351007} with a standard covering argument to deduce
\begin{equation*}
\|D^{2}v_{r,x_{0}}\|_{C^{0,\alpha}(\overline{B}_{4/5})}\le C(\|v_{r,x_{0}}\|_{L^\infty(B_{4/5})}+C);
\end{equation*}
now by applying Lemma \ref{lemo} and the maximum principle for \eqref{eq:r-nofreebdry} we obtain 
\begin{align} \label{max}
\|v_{r,x_{0}}\|_{L^\infty(B_{4/5})} &\le\|v_{r,x_{0}}\|_{L^\infty(\partial B_1)}+2C_0||f||_{L^\infty(B_1)}  \nonumber \\
&=\|u_{r,x_{0}}\|_{L^\infty(\partial B_1)}+ 2C_0||f||_{L^\infty(B_1)}\le C.
\end{align}
In particular, since $f \in C^0(B_1)$, $H(M,y):=\tilde G(D^{2}v_{r,x_0}(y)+M,y)+f(x_{0})-f(ry+x_{0})$
is continuous in $y$ on $\overline{B}_{4/5}$ and has the same ellipticity constants as $F$ (note also that $H(0,y)=0$ in $B_{4/5}$). Moreover, $w_{r,x_{0}}$ solves the equation $$H(D^{2}w_{r,x_0}(y),y)=\phi(y)\chi_{A_{r}(x_{0})} \hskip .2in y \in B_{4/5},$$ where $\phi$ has uniform bounds. %depending only on the initial data.
The operator $H$ also has interior $C^{1,1}$ estimates since it is concave. Thus, by applying Theorem 1 in \cite{MR1005611} (cf. Theorem 7.1 in \cite{MR1351007}) and
a standard covering argument (again utilizing \ref{eq:H4}), we obtain $w_{r,x_{0}}\in W^{2,p}(B_{1/2})$ for any $p>n$; selecting $p=2n$, it follows that
\begin{align}
\int_{B_{1/2}}|D^{2}w_{r,x_{0}}(y)|^{2n}\,dy & \le C(\|w_{r,x_{0}}\|_{L^{\infty}(B_{3/4})}+\|\phi \chi_{A_{r}(x_{0})}\|_{L^{2n}(B_{3/4})})^{2n} \nonumber\\
 & \le C|A_{r}(x_{0})|, \label{eq:W2n-ineq}
\end{align}
(note that the last inequality follows from (\ref{eq:ABPconsequence}) and the fact that $|A_r(x_0)| \leq |B_1|$). Since $|D^{2}u|\le K$ a.e. in $A_{r}(x_{0})$ and $$P_{r}(x_{0})=D^{2}u(ry+x_{0})-D^{2}v_{r,x_{0}}(y)-D^{2}w_{r,x_{0}}(y),$$ by utilizing \eqref{max} and \eqref{eq:W2n-ineq} we obtain  
\begin{align*}
 & |A_{r}(x_{0})\cap B_{1/2}||P_{r}(x_{0})|^{2n} =  \int_{A_{r}(x_{0})\cap B_{1/2}}|P_{r}(x_{0})|^{2n}dy\\
= & \int_{A_{r}(x_{0})\cap B_{1/2}}|D^{2}u(ry+x_{0})-D^{2}v_{r,x_{0}}(y)-D^{2}w_{r,x_{0}}(y)|^{2n}dy\\
\le & C\int_{A_{r}(x_{0})\cap B_{1/2}}|D^{2}v_{r,x_{0}}|^{2n}+|D^{2}w_{r,x_{0}}|^{2n}+|D^{2}u(ry+x_{0})|^{2n}dy\\
\le & C(|A_{r}(x_{0})\cap B_{1/2}|\|D^{2}v_{r,x_{0}}\|_{L^{\infty}(B_{1/2})}^{2n}+C|A_{r}(x_{0})|+K^{2n}|A_{r}(x_{0})\cap B_{1/2}|)\\
\le & C(|A_{r}(x_{0})\cap B_{1/2}|+|A_{r}(x_{0})|)\le C|A_{r}(x_{0})|.
\end{align*}
Next note that 
\begin{align*}
A_{r/2}(x_0) & = B_{1}\backslash((\Omega-x_{0})/(r/2)) =2 (B_{1/2}\backslash((\Omega-x_{0})/r))\\
&=2 (B_{1/2}\cap B_1 \backslash((\Omega-x_{0})/r))= 2(B_{1/2}\cap A_{r}(x_0));
\end{align*}
thus, if $|P_{r}(x_{0})|\ge(4^{n}C)^{\frac{1}{2n}}$,  
\begin{align*}
|A_{r/2}(x_{0})||P_{r}(x_{0})|^{2n} & =2^{n}|A_{r}(x_{0})\cap B_{1/2}||P_{r}(x_{0})|^{2n}\le2^{n}C|A_{r}(x_{0})|\\
 & \le\frac{|P_{r}(x_{0})|^{2n}}{2^{n}}|A_{r}(x_{0})|,
\end{align*}
which immediately gives the conclusion of the lemma.
\end{proof}
In other words, Lemma \ref{lem:fastdecay} says that the free boundary has a cusp-like behavior at $x_0$ if $|P_r(x_0)|$ is large, see Figure \ref{fig:Ar}. We now have all the ingredients to prove interior $C^{1,1}$ regularity of the solution $u$.  
\begin{figure}[ht]
\begin{centering}
\includegraphics[width=10cm]{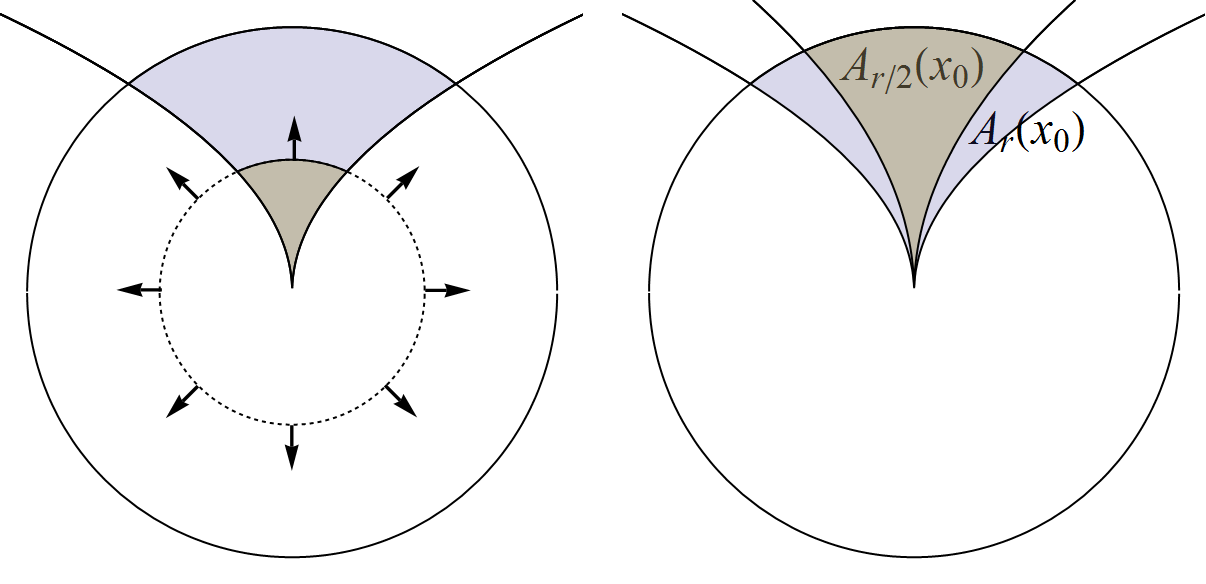}
\end{centering}
\caption{$B_r\backslash\Omega$ and $B_{r/2}\backslash \Omega$ are placed on the same scale and generate $A_r(x_0)$ and $A_{r/2}(x_0)$, respectively. Here, $x_0$ is the tip of a cusp.\label{fig:Ar}}
\end{figure}

\begin{proof}[Proof of Theorem \ref{c11}]
By assumption, $|D^{2}u|$ is bounded a.e. in $B_{1}\backslash\Omega$. Therefore, consider a point $x_{0}\in\overline{\Omega}\cap B_{1/2}$ which is a Lebesgue point
for $D^{2}u$ and where $u$ is twice differentiable (such points differ from $\Omega$ by a set of measure zero). Take $M>0$ as in Lemma \ref{lem:fastdecay}. If $\liminf_{k\to\infty}|P_{2^{-k}}(x_{0})|\le3M$,
then Lemma \ref{lem:r2bound} implies
\[
|D^{2}u(x_{0})|\le\liminf_{k\to\infty}\sup_{B_{2^{-k}}(x_{0})}\frac{2|u|}{(2^{-k})^{2}}\le2(C_{1}+3M),
\]
(recall that we may assume without loss of generality that $u(x_{0})=|\nabla u(x_{0})|=0$). In the case $\liminf_{k\to\infty}P_{2^{-k}}(x_{0})>3M$, let $k_{0}\ge3$
be such that $|P_{2^{-k_{0}-1}}|\le2M$ and $|P_{2^{-k}}|\ge2M$ for
all $k\ge k_{0}$ ($k_{0}$ can be assumed to exist by taking $M$
bigger if necessary). Then Corollary \ref{lem:BMO-consequence} implies
$|P_{2^{-k_{0}}}(x_{0})|\le2M+C_{0}$. Now let $$\overline{u}_{0}(y):=4^{k_{0}}u(2^{-k_{0}}y+x_{0})-\frac{1}{2}\langle P_{2^{-k_{0}}}(x_{0})y,y\rangle$$ and $$\tilde F(Q,y):=F(P_{2^{-k_{0}}}(x_{0})+Q, 2^{-k_{0}}y+x_{0})-f(2^{-k_{0}}y+x_{0});$$ note that $\tilde F(0,0)=0$ by the definition of $P_{2^{-k_{0}}}(x_{0})$ and $\overline{u}_{0}(y)$ solves the equation
\begin{equation} \label{eq:last}
\tilde F(D^{2}u(y),y)=\tilde f(y) \hskip .2in y \in B_{1}, 
\end{equation}
where 
$$\tilde f(y):= g(y)\chi_{A_{2^{-k_{0}}}(x_{0})},$$
and
$$g(y):=F(D^2 u(2^{-k_{0}}y+x_{0}),2^{-k_{0}}y+x_{0})-f(2^{-k_{0}}y+x_{0}) \in L^\infty(B_1),$$
with uniform bounds. %(i.e. depending only on the initial data).
Our goal is to apply Theorem 3 in \cite{MR1005611} (cf. Theorem 8.1 in \cite{MR1351007}) to \eqref{eq:last}; thus, we verify the required conditions: Lemma \ref{lem:fastdecay} implies
\[
|A_{2^{-k_{0}-j}}(x_{0})|\le2^{-jn}|A_{2^{-k_{0}}}(x_{0})|,\qquad\forall j\ge0,
\]
from which it follows that
\[
\fint_{B_{r}}|g\chi_{A_{2^{-k_{0}}}(x_{0})}|^{n}\le C\fint_{B_{r}}|\chi_{A_{2^{-k_{0}}}(x_{0})}|^{n}\le Cr^{n},\qquad\forall r\in(0,1/8);
\]
indeed, take $j$ so that $2^{-j-1}<r\le2^{-j}$ and let $A_{r}(x_{0})$
be denoted by $A_{r}$ so that 
\begin{align*}
\fint_{B_{r}}|\chi_{A_{2^{-k_{0}}}(x_{0})}|^{n} & \le\frac{|A_{2^{-k_{0}}}\cap B_{r}|}{2^{n(-j-\text{1})}}=2^{n}2^{jn}\cdot2^{-n}|2(A_{2^{-k_{0}}}\cap B_{r})|\\
 & =2^{n}2^{jn}\cdot2^{-n}|2(A_{2^{-k_{0}}}\cap B_{1/2}\cap B_{r})|\\
 & =2^{n}2^{jn}\cdot2^{-n}|2(A_{2^{-k_{0}}}\cap B_{1/2})\cap B_{2r}|\\
 & =2^{n}2^{jn}\cdot2^{-n}|A_{2^{-k_{0}-1}}\cap B_{2r}|\\
 &\le\cdots\\
 & \le2^{n}2^{jn}\cdot2^{-jn}|A_{2^{-k_{0}-j}}\cap B_{2^{j}r}|\\
 & \le2^{n}|A_{2^{-k_{0}-j}}|\le2^{n}\cdot2^{-jn}|A_{2^{-k_{0}}}|\le C2^{2n}(2^{-j-1})^{n}\\
 & \le C2^{2n}r^{n}.
\end{align*}
We are left with verifying the condition on the oscillation of $\tilde F$. To this aim, note that one may replace $\beta_{\tilde F}(y)$ by $\tilde \beta_{\tilde F}(y)$ (see e.g. (8.3) of Theorem 8.1 in \cite{MR1351007}). With this in mind, and for $P=P_{2^{-k_0}}(x_0)$,  
\begin{align*}
&\tilde \beta_{\tilde F}(y)=\sup_{Q \in \mathcal{S}} \frac{\Big |\tilde F(Q,y)-\tilde F(Q,0)\Big|}{|Q|+1}\\
&=\sup_{Q\in \mathcal{S}} \frac{\big| F(P+Q, \frac{y}{2^{k_{0}}}+x_{0})-f(\frac{y}{2^{k_{0}}}+x_{0}) -(F(P+Q,x_{0})-f(x_{0}))\big|}{|Q|+1}\\
&=\sup_{Q\in \mathcal{S}} \frac{\big| F(P+Q, \frac{y}{2^{k_{0}}}+x_{0})-F(P+Q,x_{0})+(f(x_0)-f(\frac{y}{2^{k_{0}}}+x_{0}))\big|}{|Q|+1}\\
&\leq C|y|^\alpha,
\end{align*}
(the last inequality follows from \ref{eq:H4}, the H\"older continuity of $f$, and the boundedness of $P$). Thus, the condition on the oscillation of $\tilde F$ is verified. Therefore $\overline{u}_{0}$ is $C^{2,\alpha}$ at the origin with the bound
\[
|D^{2}\overline{u}_{0}(0)|\le C
\]
for a constant $C$. This in turn implies
\[
|D^{2}u(x_{0})|\le|D^{2}\overline{u}_{0}(0)|+|P_{2^{-k_{0}}}(x_{0})|\le C,
\]
and we conclude.
\end{proof}

\section{Free boundary regularity} \label{frebd}

The aim in this section is to prove free boundary regularity for \eqref{eq:main}. In general the free boundary may develop singularities, see e.g. Schaeffer \cite{MR0516201}. Nevertheless, under a uniform thickness assumption and if $f\ge c>0$, then the free boundary is $C^1$.

\subsection{Non-degeneracy and classification of blow-ups}

The first step in the free boundary analysis is non-degeneracy (i.e. at least quadratic growth) of the solution near a free boundary point. In general this fails, even in the one-dimensional problem $u''=\chi_{\{u'' \neq 0\}}$ (see e.g. \cite[\S 3.1]{2012arXiv1212.5809F}). However, for $\{|\nabla u| \neq 0\} \subset \Omega$, non-degeneracy follows from a uniform positivity assumption on the right hand side: if $0<c\le \inf_{x \in B_1}f(x)$, then by letting $v(x):=u(x)-\frac{c|x-x_{0}|^{2}}{2n\lambda_{1}}$, one may check that $v$ is a subsolution for $F$ in $\Omega \cap B_1$ and apply the argument in \cite[Lemma 3.1]{2012arXiv1212.5809F}.    

\begin{lem}[Non-degeneracy] \label{nondeg} 
Suppose $0<c\le \inf_{x \in B_1}f(x)$ and let $u$ be a $W^{2,n}(B_{1})$ solution to \eqref{eq:main}. If $\{|\nabla u| \neq 0\} \subset \Omega$ and $x_{0}\in\overline{\Omega}\cap B_{1/2}$, then for any $r>0$ such that $B_{r}(x_{0})\Subset B_{1}$, 
\[
\sup_{\partial B_{r}(x_{0})}u\ge u(x_{0})+\frac{c}{2n\lambda_{1}}r^{2}. 
\]
\end{lem}

\noindent The previous result immediately implies a linear growth estimate on the gradient (this is usually referred to as non-degeneracy of the gradient). 
\begin{cor}
Suppose $0<c\le \inf_{x \in B_1}f(x)$ and let $u$ be a $W^{2,n}(B_{1})$ solution to \eqref{eq:main}. If $ \{|\nabla u| \neq 0\} \subset \Omega$ and $x_{0}\in\overline{\Omega}\cap B_{1/2}$, then for any $r>0$ such that $B_{r}(x_{0})\Subset B_{1}$, 
\[
\sup_{B_{r}(x_{0})}|\nabla u|\ge \frac{c}{4n\lambda_{1}}r.
\]
\end{cor}
\begin{proof}
From the non-degeneracy,
\[
\sup_{\partial B_{r}(x_{0})}u\ge u(x_{0})+\frac{c}{2n\lambda_{1}}r^{2}.
\]
Therefore there is a point $x\in\partial B_{r}(x_{0})$ such that
$u(x)-u(x_{0})\ge \frac{c}{4n\lambda_{1}}r^{2}$ . Also,
\[
u(x)-u(x_{0})\le\sup_{B_{r}(x_0)}|\nabla u||x-x_{0}|=\sup_{B_{r}(x_0)}|\nabla u|r,
\]
i.e., $\sup_{B_{r}(x_0)}|\nabla u|\ge \frac{c}{4n\lambda_{1}}r$. \end{proof}

 Non-degeneracy of the gradient and the optimal regularity result of Theorem \ref{c11} imply the porosity of the free boundary inside $B_{1/4}$, i.e. there is a $0<\delta<1$ such that every ball $B_r(x)$ contains a smaller ball $B_{\delta r}(y)$ for which $B_{\delta r}(y)\subset B_r(x)\backslash(\partial\Omega\cap B_{1/4})$.
 
\begin{lem}[Porosity of the free boundary]
 Suppose $0<c\le \inf_{x \in B_1}f(x)$ and let $u$ be a $W^{2,n}(B_{1})$ solution to \eqref{eq:main}. If $\{|\nabla u| \neq 0\} \subset \Omega$, then $\partial\Omega \cap B_{1/4}$ is porous.
\end{lem}

\begin{proof}
Let $x_{0}\in\partial\Omega\cap B_{1/4}$ and $B_{r}(x_{0})\Subset B_{1/2}$. From the non-degeneracy of the gradient,
there is a point $x\in\overline{B_{r/2}(x_0)}$ so that
\[
|\nabla u(x)|\ge Cr.
\]
Let $\bar C$ be the constant from Theorem \ref{c11} and choose $0<\delta \leq \min\{\frac{C}{2\bar C}, 1/2\}$. If $y \in B_{\delta r}(x)$, then  
\begin{align*}
|\nabla u(y)| & \ge|\nabla u(x)|-|\nabla u(y)-\nabla u(x)|\ge Cr-\|D^{2}u\|_{L^{\infty}(B_{1/2})}|x-y|\\
 & \ge Cr-\|D^{2}u\|_{L^{\infty}(B_{1/2})}\delta r\ge\frac{C}{2}r.
\end{align*}
In particular, $y \in \Omega$ and so $B_{\delta r}(x)\subset B_{r}(x_{0})\cap\Omega\subset B_{r}(x_{0})\backslash(\partial\Omega\cap B_{1/4})$.
\end{proof}

A well known consequence of the porosity is the Lebesgue negligibility of the free boundary, see e.g. \cite{PSU}.

\begin{cor}
Suppose $0<c\le \inf_{x \in B_1}f(x)$ and let $u$ be a $W^{2,n}(B_{1})$ solution to \eqref{eq:main}. If $\{|\nabla u| \neq 0\} \subset \Omega$, then $\partial\Omega$ has Lebesgue measure zero in $B_{1/4}$.
\end{cor}

\begin{lem}[Blow-up]
Suppose $0<c\le \inf_{x \in B_1}f(x)$ and let $u$ be a $W^{2,n}(B_{1})$ solution to \eqref{eq:main}, and assume $f$ to be Hölder continuous. If $\{|\nabla u|\ne0\}\subset\Omega$, then for any $x_{0}\in\partial\Omega(u) \cap B_{1/4}$ there is a sequence $\{r_{j}\}$ such that $$u_{r_{j}}(y):=\frac{u(x_{0}+r_{j}y)-u(x_{0})}{r_{j}^{2}}\to u_{0}(y)$$ as $r_{j}\to 0$ in $C_{\text{loc}}^{1,\alpha}(\mathbb{R}^{n})$, and $u_{0} \in C^{1,1}(\mathbb{R}^{n})$ solves 
\[
\begin{cases}
F(D^{2}u(y),x_{0})=f(x_{0}) & \text{a.e. in }\Omega(u_{0}),\\
|D^{2}u|\le K & \text{a.e. in }\mathbb{R}^{n}\backslash\Omega(u_{0}),
\end{cases}
\]
where $\Omega(u_{0}):=\mathbb{R}^{n}\backslash\limsup(B_{1/r_{j}}((-x_0)/r_j)\backslash\Omega(u_{r_{j}}))$, $\Omega(u_{r_{j}}):=(\Omega-x_0)/r_j$. $\Omega(u_{0})$. Moreover, $\{|\nabla u_{0}|\ne0\}\subset\Omega(u_{0})$.
\end{lem}

\begin{proof}
Theorem \ref{c11} implies $u\in C^{1,1}(B_{1/2})$; since $x_0 \in B_{1/4}$, if $r>0$ it follows that $u_{r} \in C^{1,1}(B_{1/4r})$. Let $E \Subset \mathbb{R}^n$ and note that since $C^{1,1}(E) \hookrightarrow C^{1,\alpha}(E)$ compactly for all $\alpha \in [0,1)$, there is a subsequence $\{u_{r_{j}}\}$ converging in $C_{\text{loc}}^{1,\alpha}(\mathbb{R}^{n})$ to a function $u_0 \in C^{1,1}(\mathbb{R}^{n})$ which is not identically zero by Lemma \ref{nondeg}. Thus, $|D^{2}u_{0}|$ is bounded a.e.
in $\mathbb{R}^{n}\backslash\Omega(u_{0})$ (in fact, $|D^{2}u_{0}|=0$
a.e. there since $|D^{2}u_{r_{j}}(y)|=0$ a.e. on $\{|\nabla u_{r_{j}}|=0\}$). Next, let  $y \in\Omega(u_{0})$
%be a point where $u_0$ is twice-differentiable ($u_0$ is twice differentiable almost everywhere). 
and select $\delta>0$ such that $B_{\delta}(y_{0})\subset\Omega(u_{r_{j}})$ for $j$ large enough (by taking a further subsequence, if necessary); note that $u_{r_{j}}$ is $C^{2,\alpha}(B_\delta (y_0))$ in this set (by \cite[Theorem 8.1]{MR1351007}). We can therefore, without loss of generality, assume strong convergence of $u_{r_{j}}$
to $u_{0}$ in $C^{2}(B_{\delta}(y_{0}))$. In particular,
\begin{align*}
F(D^{2}u_{0}(y),x_{0}) & =\lim_{j\to\infty}F(D^{2}u_{r_{j}}(y),x_{0}+r_j y)\\
 & =\lim_{j\to\infty}f(x_{0}+r_j y)=f(x_{0}),\qquad y\in B_{\delta}(y_{0}).
\end{align*}
To conclude the proof, note that for $j$ large enough, $|\nabla u_{r_{j}}|\ne0$ in a neighborhood of a point $x$ where $|\nabla u_{0}(x)|\ne0$, and so it follows that $\{|\nabla u_{0}|\ne0\}\subset\Omega(u_{0})$.
\end{proof}

Since blow-up solutions are solutions to a free boundary problem on $\mathbb{R}^n$, one may consider the classification of these global solutions. To this aim, one introduces $$\delta_r(u,x):= \frac{\MD(\lambda \cap B_r(x))}{r},$$ where $\lambda:= B_1 \setminus \Omega$ (recall that $\MD(E)$ is the smallest possible distance between two hyperplanes containing $E$). Note that $\delta$ is well-behaved under scaling and thus with respect to the blow-up procedure: $\delta_1(u_r,0)=\delta_r(u,x)$, where $u_r(y)=(u(x+ry)-u(x))/r^2$. Now after blow-up,  even for general operators, the operator will solely be a function of the matrix variable and if $f$ is a positive function bounded away from zero, by letting $G(M):=F(M,x_{0})/f(x_{0})$, the problem of classifying global solutions reduces to the content of \cite[Proposition 3.2]{2012arXiv1212.5809F}. 

\begin{prop} \label{hs}
Suppose $0<c\le \inf_{x \in B_1}f(x)$, fix $x_0 \in B_1$, and let $u_{0}$ be a $W^{2,n}(\mathbb{R}^{n})$ solution to 
\[
\begin{cases}
F(D^{2}u(y),x_{0})=f(x_{0}) & \text{a.e. in }\Omega(u_{0}),\\
|D^{2}u|\le K & \text{a.e. in }\mathbb{R}^{n}\backslash\Omega(u_{0}),
\end{cases}
\]
with $\{|\nabla u_{0}|\ne0\}\subset\Omega(u_{0})$. If $F$ is convex and there exists $\epsilon_{0}>0$ such that
\[
\delta_{r}(u,x)\ge\epsilon_{0},\qquad\forall r>0,\,\forall x\in\partial\Omega(u_0),
\]
then $u_{0}$ is a half-space solution, $u_{0}(x)=\gamma_{x_{0}}[(x\cdot e_{x_{0}})^{+}]^{2}/2+c$,
where $e_{x_{0}}\in\mathbb{S}^n$ and $\gamma_{x_{0}}\in(1/\lambda_{1},1/\lambda_{0})$
are such that $F(\gamma_{x_{0}}e_{x_{0}}\otimes e_{x_{0}},x_{0})=f(x_{0})$.
\end{prop}

\subsection{Directional monotonicity and \texorpdfstring{$C^1$}{C1,1} regularity of the free boundary}

In what follows, two technical monotonicity lemmas will be established and utilized in proving that the free boundary is $C^1$. 

\begin{lem} \label{gain}
Let $u$ be a $W^{2,n}(B_{1})$ solution of 
\begin{equation}
\begin{cases}
F(D^{2}u(x),rx)=f(rx) & \text{a.e. in }B_{1}\cap\Omega,\\
|D^{2}u|\le K & \text{a.e. in }B_{1}\backslash\Omega,
\end{cases}\label{eq:scaled}
\end{equation}
and assume $f$ is $C^{0,1}$, $\inf_{B_1} f>0$, and $F$ is convex in the matrix variable and satisfies \ref{hyp:0@0}, \ref{hyp:uniformellipticity}, and \ref{eq:H4} with $\alpha = 1$. If $\{u\ne0\} \subset \Omega$ and $C_{0}\partial_{e}u-u\ge-\epsilon_{0}$ in $B_1$,
then 
\[
C_{0}\partial_{e}u-u\ge0
\]
 in $B_{1/2}$ provided $$\epsilon_{0}\le (\inf_{B_1} f)/(64n\lambda_{1}),$$ and $$0<r \le\min\{\|f\|_{L^{\infty}(B_{1})}/(2C_{0}\|\nabla f\|_{L^{\infty}(B_{1})}+2C_{0}\overline{C}),1\}.$$
\end{lem}
\begin{proof}
Let $x \in \Omega$ and $\partial F(M,x)$ denote the subdifferential of $F$ at the point $(M,x)$ and note that convexity implies $\partial F(M,x) \neq \emptyset$. Consider a measurable function $P^M$ mapping $(M,x)$ to $P^{M}(x)\in\partial F(M,x)$. Since $u\in C_{\text{loc}}^{2,\alpha}(\Omega)$ (see e.g. \cite[Theorem 8.1]{MR1351007}), we can define the measurable coefficients $a_{ij}(x):=(P^{D^{2}u(x)}(rx))_{ij}\in\partial F(D^{2}u(x),rx)$. By convexity of $F(\cdot,x)$ and the fact that $F(0,x)\equiv0$, we have
\begin{equation}
a_{ij}(x)\frac{\partial_{ij}u(x+he)-\partial_{ij}u(x)}{h}\le\frac{F(D^{2}u(x+he),rx)-F(D^{2}u(x),rx)}{h},\label{eq:convexityconsequence1}
\end{equation}
\[
a_{ij}(x)\partial_{ij}u(x)=F(0,rx)+a_{ij}(x)\partial_{ij}u\ge F(D^{2}u(x),rx)=f(rx),
\]
provided $x+h\in\Omega$. Note that by \eqref{eq:convexityconsequence1} and \cite[Theorem 3.8]{CCKS} (uniform limits of viscosity solutions are viscosity solutions),  we have 
\begin{align*}
 &a_{ij}(x)\partial_{ij}\partial_{e}u(x) \le\limsup_{h\to0}\frac{F(D^{2}u(x+he),rx)-F(D^{2}u(x),rx)}{h}\\
 & =\limsup_{h\to0}\frac{F(D^{2}u(x+he),rx)-f(rx)}{h}\\
 & =\limsup_{h\to0}\frac{F(D^{2}u(x+he),rx)-F(D^{2}u(x+he),rx+rhe)}{h}\\
 &\qquad\qquad\qquad +\frac{f(rx+rhe)-f(rx)}{h}\\
 & =r\limsup_{h\to0}\frac{F(D^{2}u(x+he),rx)-F(D^{2}u(x+he),rx+rhe)}{h}\\
 &\qquad\qquad\qquad +\frac{f(rx+rhe)-f(rx)}{rh}\\
 & =r(\partial_{e}f)(rx)-r(\partial_{x,e}F)(D^{2}u(x),rx) \qquad\text{a.e.},
\end{align*}
where $\partial_{x,e}$ denotes the spatial directional derivative in
the direction $e$. If there is $y_{0} \in B_{1/2} \cap \Omega$ such that
$C_{0}\partial_{e}u(y_{0})-u(y_{0})<0$, then consider the auxiliary
function
\[
w(x)=C_{0}\partial_{e}u(x)-u(x)+c\frac{|x-y_{0}|^{2}}{2n\lambda_{1}},
\]
where $c=\inf_{B_{1}}f/2$. Note that for $r\le\min\{c/(C_{0}\|\nabla f\|_{L^{\infty}(B_{1})}+C_{0}\overline{C}),1\}$,
\begin{align*}
a_{ij}(x)\partial_{ij}w(x) & \le rC_{0}(\partial_{e}f)(rx)-rC_{0}(\partial_{x,e}F)(D^{2}u(x),rx)-f(rx)+c\\
 & \le rC_{0}\|\nabla f\|_{L^{\infty}(B_{1})}+rC_{0}\overline{C}-f(rx)+c\le2c-f(rx)\le0.
\end{align*}
Hence $w$ is a supersolution and therefore attains its minimum on the
boundary of $B_{1/4}(y_0)\cap\Omega$. However on $\partial\Omega$,
$w$ is positive (since both $u$ and $\partial_{e}u$ are zero); thus, the minimum is attained on $\partial B_{1/4}(y_0)$, and this implies
\[
0>\min_{B_{1/4}(y_0)\cap\Omega}w\ge-\epsilon_{0}+\frac{c}{32n\lambda_{1}},
\]
a contradiction if $\epsilon_{0}\le c/(32n\lambda_{1})$. \end{proof}

\begin{lem} \label{gain2}
Let $u$ be a $W^{2,n}(B_{1})$ solution of \eqref{eq:scaled} where $F$ and $f$ are $C^{0,1}$ in the spatial variable, $F$ is $C^1$ in the matrix variable, and $F$ is convex and satisfies \ref{hyp:0@0} - \ref{hyp:uniformellipticity}. Assume further that $\{\nabla u\ne0\} \subset \Omega$ and $\inf_{B_1} f >0$. If $C_{0}\partial_{e}u-|\nabla u|^{2}\ge-\epsilon_{0}$ in $B_1$
for some $C_{0}, \epsilon_{0}>0$, then 
\[
C_{0}\partial_{e}u-|\nabla u|^{2}\ge0
\]
in $B_{1/2}$ provided that $\epsilon_{0}\le \mu_1$
%\frac{c^{2}\lambda_{0}}{32n^{3}\lambda_{1}^{3}}$
and $0<r\le \mu_2$, where $\mu_1>0$ and $\mu_2>0$ are constants depending on given bounds.
%\frac{c^{2}\lambda_{0}}{4n^{2}\lambda_{1}^{2}\overline{C}_{1}}$.
%Here $\overline{C}_{1}=\overline{C}_{1}(u)$.
\end{lem}
\begin{proof}
By differentiating \eqref{eq:scaled}, it follows that
\[
F_{ij}(D^{2}u(y),ry)\partial_{ij}\nabla u=r\nabla f(ry)-r\nabla_{x}F(D^{2}u(y),ry),\qquad\text{a.e in }\Omega.
\]
Since $u\in C_{\text{loc}}^{2,\alpha}(\Omega)$ (by \cite[Theorem 8.1]{MR1351007}) and the right hand side of the equation above is in $L^\infty(\Omega)$ (hence, $L^{p}(\Omega)$ for any $p>0$), it follows by elliptic regularity theory that $\nabla u \in W_{\text{loc}}^{2,p}(\Omega)$ for any $p<\infty$ (see e.g. \cite[Corollary 9.18]{gilbarg:01}). By applying the operator $F_{ij}(D^{2}u(y),ry)\partial_{ij}$
to $|\nabla u|^{2}$, we obtain 
\begin{align}
&F_{ij}(D^{2}u(y),ry)\partial_{ij}|\nabla u(y)|^{2} \nonumber \\
&=2F_{ij}(D^{2}u(y),ry)\partial_{ijk}u(y)\partial_{k}u(y) +2F_{ij}(D^{2}u(y),ry)\partial_{ik}u(y)\partial_{jk}u(y)\label{eq:diffnablasquared} \\
 & =2r(\nabla f(ry)-\nabla_{x}F(D^{2}u(y),ry))\cdot\nabla u(y)+2F_{ij}(D^{2}u(y),ry)\partial_{ik}u(y)\partial_{jk}u(y)\nonumber 
\end{align}
For differentiable operators, the ellipticity condition can be written as
$$F_{ij}(D^{2}u(y),ry)\xi_{i}\xi_{j}\ge\lambda_{0}|\xi|^{2};$$ thus, \eqref{eq:diffnablasquared} yields
\begin{align}\label{eq:nablasquareest}
F_{ij}(D^{2}u(y),ry)\partial_{ij} & |\nabla u(y)|^{2} \\ 
&\ge 2r(\nabla f(ry)-\nabla_{x}F(D^{2}u(y),ry))\cdot\nabla u(y)+2\lambda_{0}|D^{2}u(y)|^{2}.\nonumber
\end{align}
Now \ref{hyp:0@0}-\ref{hyp:uniformellipticity} and the positivity of $f$ imply
\begin{equation}
0<c\le f(ry)=|F(D^{2}u(y),ry)-F(0,ry)|\le2n\lambda_{1}|D^{2}u|, \label{eq:D2uboundfrombelow}
\end{equation}
where $c:= \inf_{B_1} f$. By combining \eqref{eq:nablasquareest}
and \eqref{eq:D2uboundfrombelow}, it follows that
\[
F_{ij}(D^{2}u(y),ry)\partial_{ij}|\nabla u(y)|^{2}\ge2r(\nabla f(ry)-\nabla_{x}F(D^{2}u(y),ry))\cdot\nabla u(y)+\frac{c^{2}\lambda_{0}}{2n^{2}\lambda_{1}^{2}}.
\]
The proof now follows as in Lemma \ref{gain}: assume by contradiction that there is a point $y_{0} \in B_{1/2} \cap \Omega$ such that $C_{0}\partial_{e}u(y_{0})-|\nabla u(y_{0})|^{2}<0$ (outside $\Omega$ we have $|\nabla u|=0$). Let $d=\frac{c^{2}\lambda_{0}}{4n^{2}\lambda_{1}^{2}}$ and $$w(y)=C_{0}\partial_{e}u(y)-|\nabla u(y)|^{2}+d\frac{|y-y_{0}|^{2}}{2n\lambda_{1}}.$$ 
Next note that for $r$ sufficiently small, $w$ is a supersolution of $F_{ij}(D^{2}u(y),ry)\partial_{ij}$. Indeed, 
\begin{align*}
&F_{ij}(D^{2}u(y),ry)\partial_{ij}w  \\
&\le rC_{0}\|\nabla f\|_{L^{\infty}(B_{1})}+rC_{0}\overline{C}\\
&\qquad -2r(\nabla f(ry)-\nabla_{x}F(D^{2}u(y),ry))\cdot\nabla u(y)-\frac{c^{2}\lambda_{0}}{2n^{2}\lambda_{1}^{2}}+d \\
&\le r\overline{C}_{1}(u)-\frac{c^{2}\lambda_{0}}{2n^{2}\lambda_{1}^{2}}+d\le0,
\end{align*}
where for the last inequality we require $r\le\frac{c^{2}\lambda_{0}}{4n^{2}\lambda_{1}^{2}\overline{C}_{1}}$. Therefore $w$ attains a minimum on the boundary of $B_{1/4}(y_0)\cap\Omega$.
However, on $\partial\Omega$, $w$ is non-negative since both $u$
and $\partial_{e}u$ are zero, so the minimum has to be attained on $\partial B_{1/4}(y_0)$,
and this implies 
\[
0>\min_{B_{1/4}(y_0)\cap\Omega}w\ge-\epsilon_{0}+\frac{d}{32n\lambda_{1}}
\]
which is a contradiction if $\epsilon_{0}\le d/(32n\lambda_{1})$.
%=\frac{c^{2}\lambda_{0}}{32n^{3}\lambda_{1}^{3}}$. 
\end{proof}

\noindent We are now in a position to prove that under a suitable thickness assumption, the free boundary is $C^1$.
\begin{thm} \label{fbreg}
Let $u:B_{1}\to\mathbb{R}$ be a $W^{2,n}(B_{1})$ solution of \eqref{eq:main}. Let $F$ be a convex operator satisfying \ref{hyp:0@0}, \ref{hyp:uniformellipticity}, and \ref{eq:H4} with $\alpha = 1$, and assume further that $f$ is $C^{0,1}$. If $\{u\ne0\}\subset\Omega$
and there exists $\epsilon>0$ such that
\[
\delta_{r}(u,x)>\epsilon,\qquad\forall r<1/4,\, x\in\partial\Omega\cap B_{r},
\]
then there exists $r_{0}>0$ depending only on $\epsilon$ and given bounds such that $\partial\Omega\cap B_{r_{0}}(x)$ is a $C^{1}$-graph.
\end{thm}

\begin{proof}
Let $x\in\partial\Omega\cap B_{1/8}$ and consider the rescaling
$u_{r}(y):=\frac{u(ry+x)-u(x)}{r^{2}}$. By Theorem \ref{c11} we have
a uniform $C^{1,1}$-estimate with respect to $r$ and can therefore
find a subsequence $\{u_{r_{j}}\}$ converging in $C_{\text{loc}}^{1}(\mathbb{R}^n)$
to a global solution $u_{0}$, where $u_{0}(0)=0$. The thickness assumption implies $\delta_{r}(u,x)>\epsilon$ for all $r>0$, hence $u_{0}(y)=\gamma\frac{((y\cdot e_{x})_{+})^{2}}{2}$
according to Proposition \ref{hs}, where $\gamma\in[\lambda_{0},\lambda_{1}]$
and $e_{x}\in\partial B_{1}$. Now let $0<s\le1$. Then 
\[
\frac{\partial_{e}u_{0}}{s}-u_{0}\ge0
\]
in $B_{1}$ for any direction $e\in\partial B_{1}$ such that $e\cdot e_{x}\ge s$.
From the $C^{1}$-convergence of $\{u_{r_{j}}\}$ we have

\[
\frac{\partial_{e}u_{r_{j}}}{s}-u_{r_{j}}\ge-\epsilon_{0}
\]
in $B_{1}$ for $j\ge k(s,x)$ and $\epsilon_{0}$ as in Lemma \ref{gain}.
Therefore $u_{r_{j}}$ fulfills the assumptions of this lemma and
the above inequality can be improved to
\begin{equation}
\frac{\partial_{e}u_{r_{j}}(y)}{s}-u_{r_{j}}(y)\ge0,\qquad y\in B_{1/2}.\label{eq:wge0}
\end{equation}
For $s=1$, i.e. $e=e_{x}$, multiplying \eqref{eq:wge0} by $\exp(e\cdot y)$ implies
\[
\partial_{e}[\exp(-e\cdot y)u_{r_{j}}(y)]=\exp(-e\cdot y)(\partial_{e}u_{r_{j}}(y)-u_{r_{j}}(y))\ge0.
\]
Integrating this expression yields
\[
\exp(-e\cdot y)u_{r_{j}}(y)-\underbrace{u_{r_{j}}(0)}_{=0}=\int_{0}^{e\cdot y}\partial_{e}[\exp(-e\cdot z)u_{r_{j}}(z)]d(e\cdot z)\ge0,
\]
so $u_{r_{j}}(y)\ge0$ in $B_{1/2}$ and $\partial_{e}u_{r_{j}}(y)\ge 0$
follows from \eqref{eq:wge0}. In particular, we have shown that if $x\in\partial\Omega\cap B_{1/8}$ and $e\cdot e_{x}\ge s$, then $\partial_e u(z) \ge 0 $ for all $z \in B_{r_j/2}(x)$, where $r_j=r_j(s,x)$. Now $$\partial\Omega\cap \overline{B}_{1/16} \subset \bigcup_{x \in \partial\Omega\cap \overline{B}_{1/16}} B_{r_j/2}(x),$$ so by extracting a finite subcover and relabeling the radii, it follows that $$\partial\Omega\cap \overline{B}_{1/16} \subset \bigcup_{k=1}^N B_{\eta_k}(x_k),$$ where $\eta_k=\eta_k(x_k,s)$; set $\eta=\eta(s):=\displaystyle \min_k \eta_k$. Thus, for all $x \in \partial\Omega\cap \overline{B}_{1/16}$, we have $\partial_e u(z) \ge 0$ for all $z \in B_\eta(x)$, where $\eta$ only depends on $s$ and the given data (via the $C^1$ convergence of $u_{r_j}$). Therefore, if $s_0 \in (0,1)$, by letting $r_0:=\eta(s_0)$, it follows that the free boundary $\partial\Omega\cap B_{r_0}(x)$ is $s_0$-Lipschitz. Moreover, note that in a small neighborhood of the origin, by picking $s$ sufficiently small, the Lipschitz constant of the free boundary can be made arbitrarily small (the neighborhood only depends on $\eta(s)$). This shows that the free boundary is $C^1$ at the origin, and the same reasoning applies to any other point in $\partial\Omega\cap \overline{B}_{r_0}(x)$.
\end{proof}

\begin{rem}
In view of Lemma \ref{gain2}, we can replace the condition $\{u\ne0\}\subset\Omega$ by $\{\nabla u\ne0\}\subset\Omega$ in Theorem \ref{fbreg} whenever $F$ is $C^1$ in the matrix variable.  
\end{rem}

\begin{rem}
The free boundary analysis remains valid for more general operators, e.g. such as the ones appearing in Corollary \ref{genops}.
\end{rem}

\section{Parabolic case} \label{para}
In this section we generalize the former results regarding optimal regularity of the solution as well as $C^1$ regularity of the free boundary to the non-stationary setting. Since the parabolic case is very similar to the elliptic one, we mostly outline the proofs. The setup of the problem is as follows.  

\begin{itemize}
\item Let $Q_{r}(X):=B_{r}(x)\times(t-r^2,t)$, where $X=(x,t)$. For
convenience, $Q_{r}:=Q_{r}(0)$.
\item Instead of \eqref{eq:main} we consider the following problem,
\begin{equation}
\begin{cases}
\mathcal{H}(u(X),X)=f(X) & \text{a.e. in }Q_{1}\cap\Omega,\\
|D^{2}u|\le K & \text{a.e. in }Q_{1}\backslash\Omega,
\end{cases}\label{eq:paramain}
\end{equation}
where $\mathcal{H}(u(X),X):=F(D^{2}u(X),X)-\partial_{t}u(X)$, $\Omega\subset Q_{1}$
is some unknown set, and $K$ is a positive constant as before. We still assume $F$ to satisfy \ref{hyp:0@0}-\ref{hyp:concavity}
for all $X\in Q_{1}$ and 
\begin{align}\label{eq:paraH4}
	F(M,x,t)-F(M,y,s)\le C|M|(|x-y|^{\alpha_1}+|t-s|^{\alpha_2}).
\end{align}
with $\alpha_1,\alpha_2\in (0,1]$.
\item
We assume $f$ to be at least Hölder continuous in both the spatial and time coordinates.
\item Let $A_{r}(X^{0}):=\{(x,t)\in Q_{1}:(rx,r^{2}t)\in Q_{r}\backslash\Omega\}$.
\item Let $\tilde{D}^{2}u:=(D_{x}^{2}u,D_{t}u)$ denote the parabolic Hessian.
\item Let 
\[
\delta_{r}(u,X^{0}):=\inf_{t\in[t_{0}-r^{2},t_{0}+r^{2}]}\frac{\MD(\Proj_{x}(A\cap(B_{r}(x^{0})\times\{t\})))}{r},
\]
where $\MD(E)$ stands for the minimal diameter, i.e., the smallest
distance between two parallel hyperplanes that trap the set $E$,
and $\Proj_{x}$ is the projection on the spatial coordinates.
\end{itemize}
The main theorems corresponding to Theorem \ref{c11} and \ref{fbreg} are now stated
for the parabolic case; the first giving the optimal regularity of
solutions.
\begin{thm}[Interior $C_x^{1,1}\cap C_t^{0,1}$ regularity]
\label{thm:pararegularity}
Let $u:Q_{1}\to\mathbb{R}$ be a $W_{x}^{2,n}\cap W_{t}^{1,n}$ solution
of \eqref{eq:paramain}. Then there is a constant $C=C(n,\lambda_{0},\lambda_{1},\|u\|_{\infty},\|f\|_{\infty})>0$
such that
\[
|\tilde{D}^{2}u|\le C,\qquad\text{in }Q_{1/2}.
\]
\end{thm}
The second theorem gives $C^{1}$ regularity of the free boundary if we add some additional assumptions on $\delta_{r}$, $f$ and $F$, as in the elliptic setting.
\begin{thm}[$C^{1}$ regularity of the free boundary]
\label{thm:paraboundaryregularity}
Let $u:Q_{1}\to\mathbb{R}$ be a $W_{x}^{2,n}\cap W_{t}^{1,n}$ solution of \eqref{eq:paramain}, and assume $\{u\ne 0\}\subset\Omega$. Suppose that $f$ is Lipschitz in $(x,t)$ and $f\ge c>0$. Let $F$ be convex in the matrix variable and suppose $F$ satisfies \ref{hyp:0@0}, \ref{hyp:uniformellipticity}, and \eqref{eq:paraH4} with $\alpha_1=\alpha_2 =1$. Then there exists an $\epsilon>0$ such that if
\[
	\delta_{r}(u,X^{0})>\epsilon
\]
uniformly in $r$ and $X^{0}\in\partial\Omega\cap Q_{r}$, then $\partial\Omega\cap Q_{r_{0}}$
is a $C^{1}$-graph in space-time, where $r_{0}$ depends only on $\epsilon$ and the data. 
\end{thm}
Theorem \ref{thm:pararegularity} follows from results corresponding to \cite[Lemma 2.1 and Proposition 2.2]{2013arXiv1309.0782F} which readily generalize to the parabolic setting thanks to our results in the ellipic case and \cite[Remark 6.3]{2013arXiv1309.0782F}. Indeed, we can show the inequality 
\[
\sup_{Q_{r}(0)}|u-P_{r}|\le Cr^{2},\qquad r\in(0,1)
\]
for some parabolic polynomials $P_{r}$ that solve the homogeneous equation \[\mathcal{H}(P_{r},0)=0,\] a result that is in the same vein as Lemma \ref{lemo}. Moreover, the above inequality together with an argument similar to the proof of Lemma \ref{lem:fastdecay} imply the geometric decay of the coincidence sets,
\[
|A_{r/2}|\le\frac{|A_{r}|}{2^{n+1}}.
\]
Theorem \ref{thm:pararegularity} is then proven in the same way as in the elliptic case.

Regarding the regularity of the free boundary, Lemma \ref{nondeg} is easily generalized
since the maximum principle holds in our case as well (see \cite[Corollary 3.20]{W1}), and the rest of the results are extended with the following parabolic blow-up lemma.

\begin{lem}\label{lem:blowuppara}
%Suppose $0<c\le \inf_{(x,t) \in Q_1}f(x,t)$ and 
Let $u$ be a $W_x^{2,n}\cap W_t^{1,n}$ solution to \eqref{eq:paramain}. If $\{u\ne0\}\subset \Omega$, then for any $(x_{0},t_0) \in\partial\Omega(u) \cap Q_{1/4}$ there is a sequence $\{r_{j}\}$ such that $$u_{r_{j}}(y,t):=\frac{u(x_{0}+r_{j}y, t_0+r_j^2t)-u(x_{0},t_0)}{r_{j}^{2}}\to u_{0}(y,t)$$ locally uniformly as $r_{j}\to 0$, and $u_{0}$ solves 
\[
\begin{cases}
F(D^{2}u(y,t),x_{0},t_0)-\partial_t u(y,t)=f(x_{0}, t_0) & \text{a.e. in }\Omega(u_{0}),\\
|\tilde D^{2}u|\le K & \text{a.e. in }\mathbb{R}^{n}\backslash\Omega(u_{0}),
\end{cases}
\]
where $\{u_{0}\ne0\}\subset\Omega(u_{0})$.
\end{lem}

\begin{proof}
By $C_x^{1,1} \cap C_t^{0,1}$ regularity of $u$ and the fact that $u=0$ on $\partial \Omega$, it follows that the sequence $\{u_{r_j}\}$ is uniformly bounded; hence, up to a subsequence, $u_j \rightarrow u_0$ locally uniformly. 
Define $\Omega(u_{0})$ to be the limit of the open sets $\Omega_{j}:=\{(x,t):(x_0+r_jx, t_0+r_j^2t)\in \Omega\}$ (as in the elliptic case), and note that $\tilde D^{2}u$ is bounded on the complement of $\Omega(u_0)$ (since $\{u\ne0\}\subset \Omega$). Moreover, $u_0$ is not identically zero by non-degeneracy. Next, let $(y,t) \in\Omega(u_{0})$ and select $\delta>0$ such that $Q_{\delta}(y,t)\subset \Omega_{j}$ for $j$ large enough; note that $u_{r_{j}}$ is $C_x^{2,\alpha} \cap C_t^{1,\alpha}$ in this set (by the parabolic Evans-Krylov theorem \cite{MR0262675}). We can therefore, without loss of generality, assume $C_x^{2} \cap C_t^{1}$ convergence of $u_{r_{j}}$ to $u_{0}$ in $Q_{\delta}(y,t)$. In particular,  
\begin{align*}
F(D^{2}u_{0}(y,t),x_{0}, t_0) & =\lim_{j\to\infty} \Big(F(D^{2}u_{r_{j}}(y,t),x_{0}+r_j y, t_0+r_j^2t) - \partial_t u_{r_{j}}(y,t) \Big)\\
 & =\lim_{j\to\infty}f(x_{0}+r_j y, t_0+r_j^2t)=f(x_{0},t_0),\qquad y\in Q_{\delta}(y,t).
\end{align*}
To conclude the proof, note that for $j$ large enough, $u_{r_{j}}\ne 0$ in a neighborhood of a point $(y,t)$ where $u_{0}(y,t)\ne0$, and so it follows that $\{u_{0}\ne0\}\subset\Omega(u_{0})$.

\end{proof}

\noindent Since blow-up solutions are solutions to a free boundary problem on $\mathbb{R}^{n+1}$, one may consider the classification of these global solutions just like in the elliptic case. By letting $\mathcal{G}(M):=\mathcal{H}(M,x_0,t_0)/f(x_0,t_0)$, the problem reduces to the content of \cite[Proposition 3.2]{2013arXiv1309.0782F}. 

\begin{prop} \label{hspara}
Fix $X_0:=(x_0,t_0)$. If $u_0$ is a solution to
\[
\begin{cases}
\mathcal{H}(D^{2}u(y),X_{0})=f(X_{0}) & \text{a.e. in }\Omega(u_{0}),\\
|D^{2}u|\le K & \text{a.e. in }\mathbb{R}^{n}\backslash\Omega(u_{0}),
\end{cases}
\]
with $\{u_{0} \ne0\}\subset\Omega(u_{0})$, and there exists $\epsilon_{0}>0$ such that
\[
\delta_{r}(u,x)\ge\epsilon_{0},\qquad\forall r>0,\,\forall x\in\partial\Omega(u_0),
\]
then $u_{0}$ is time-independent and of the form $u_{0}(x)=\gamma_{X_{0}}[(x\cdot e_{X_{0}})^{+}]^{2}/2$,
where $e_{X_{0}}\in\mathbb{S}^n$ and $\gamma_{X_{0}}\in(1/\lambda_{1},1/\lambda_{0})$
are such that $F(\gamma_{X_{0}}e_{X_{0}}\otimes e_{X_{0}}, X_{0})=f(X_{0})$.
\end{prop}
This proposition can, in turn, be used to prove that the time derivative $\partial_t u$ vanishes on the free boundary. The proof follows the same line as \cite{2013arXiv1309.0782F} except that Proposition 3.2 is replaced in their proof with Proposition \ref{hspara}. The result is stated in the following lemma.
\begin{lem}
	Let $u$, $f$, $F$ and $\delta_r$ be as in Theorem \ref{thm:paraboundaryregularity} and $\{u\ne 0\}\subset\Omega$. Then
	\[
		\lim_{\Omega\ni X\to \partial\Omega}\partial_t u(X)=0
	\]
\end{lem}
The parabolic counterpart of Lemma \ref{gain} follows by replacing $w$ given in that proof with
\[
C\partial_e u(X)-u(X)+\tilde c\frac{|x-x_0|^2-(t-t_0)}{2n\lambda_1+1},
\]
where $\tilde c:=\inf_{Q_{1}}f/2$; this is where the Lipschitz assumptions on $F$ and $f$ come into play. With this in mind, the proof of Theorem \ref{thm:paraboundaryregularity} follows as in the elliptic case. 
\bibliographystyle{amsalpha}
\bibliography{References}

\signei

\signam

\end{document}